\documentclass[12pt]{article}
\usepackage[german, french, american]{babel}
\usepackage[utf8]{inputenc}
\usepackage{amsmath, amssymb, amsfonts}
\usepackage{amsthm}
\usepackage{mathtools}
\usepackage{mathrsfs}
\usepackage{enumerate}
\usepackage{graphicx}
\usepackage{color}

\usepackage{thumbpdf}
\usepackage{hyperref}

\newtheorem{theorem}{Theorem}[section]
\newtheorem{lemma}[theorem]{Lemma}

\newtheorem{corollary}[theorem]{Corollary}
\newtheorem{definition}[theorem]{Definition}

\numberwithin{equation}{section}

\numberwithin{equation}{section}

\usepackage{fix-cm}
\usepackage{epstopdf}

\hypersetup{
	pdftitle    = {On the Cauchy Problem for a Linear Harmonic Oscillator with Pure Delay},
	pdfsubject  = {Harmonic Oscillator with Pure Delay},
	pdfauthor   = {Denys Ya. Khusainov, Michael Pokojovy, Elvin I. Azizbayov},
	pdfkeywords = {functional-differential equations, harmonic oscillator, pure delay, well-posedness, solution representation}
}

\begin{document}

\title{On the Cauchy Problem for a Linear Harmonic Oscillator with Pure Delay}

\author{
Denys Khusainov\thanks{Faculty of Cybernetics, Kyiv National Taras Shevchenko University, Kyiv, Ukraine \hfill
\texttt{d.y.khusainov@gmail.com}} \and
Michael Pokojovy\thanks{Department of Mathematics and Statistics, University of Konstanz, Konstanz, Germany \hfill \texttt{michael.pokojovy@uni-konstanz.de}} \and
Elvin Azizbayov\thanks{Faculty of Mechanics and Mathematics, Baku State University, Azerbaijan \hfill
\texttt{eazizbayov@bsu.az}}
}

\date{\today}

\maketitle

\begin{abstract}
	In the present paper, we consider a Cauchy problem for a linear second order in time abstract differential equation with pure delay.
	In the absence of delay, this problem, known as the harmonic oscillator, has a two-dimensional eigenspace 
	so that the solution of the homogeneous problem can be written as a linear combination of these two eigenfunctions. 
	As opposed to that, in the presence even of a small delay, the spectrum is infinite and a finite sum representation is not possible.
	Using a special function referred to as the delay exponential function,
	we give an explicit solution representation for the Cauchy problem associated with the linear oscillator with pure delay. 
	In contrast to earlier works, no positivity conditions are imposed.
\end{abstract}


{\bf AMS: }
 	34K06, 
	34K26, 
	39A06, 
 	39B42  

\pagestyle{myheadings}
\thispagestyle{plain}
\markboth{\textsc{D. Khusainov, M. Pokojovy, A. Azizbayov}}{\textsc{Linear Harmonic Oscillator with Pure Delay}}

\section{Introduction}
Let $X$ be a (real or complex) Banach space and let $x(t) \in X$ describe the state of a physical system at time $t \geq 0$. 
With $a(t) = \ddot{x}(t)$ denoting the acceleration of system, the Newton's second law of motion states that
\begin{equation}
	F(t) = Ma(t) \text{ for } t \geq 0, \label{EQUATION_NEWTON_SECOND_LAW}
\end{equation}
where $M \colon D(M) \subset X \to X$ is a linear, continuously invertible, accretive operator
representing the ``mass'' of the system. 
When being displaced from its equilibrium situated in the origin, the system is affected by a restoring force $F(t)$. 
In classical mechanics, this force is postulated to be proportional to the instantaneous displacement, i.e.,
\begin{equation}
	F(t) = Kx(t) \text{ for } t \geq 0 \label{EQUATION_MATERIAL_LAW}
\end{equation}
for some closed, linear operator $K \colon D(K) \subset X \to X$.
When $M^{-1} K$ is a bounded linear operator,
plugging Equation (\ref{EQUATION_MATERIAL_LAW}) into (\ref{EQUATION_NEWTON_SECOND_LAW}),
we arrive at the classical harmonic oscillator model
\begin{equation}
	\ddot{x}(t) = M^{-1} K x(t) \text{ for } t \geq 0. \label{EQUATION_HARMONIC_OSCILLATOR}
\end{equation}

Assuming now that the restoring force is proportional to the value of the system at some past time $t - \tau$, 
Equation (\ref{EQUATION_MATERIAL_LAW}) is replaced with the relation
\begin{equation}
	F(t) = K x(t - \tau) \text{ for } t \geq 0, \label{EQUATION_MATERIAL_LAW_WITH_DELAY}
\end{equation}
where $\tau > 0$ is a time delay. Plugging Equation (\ref{EQUATION_MATERIAL_LAW_WITH_DELAY}) into (\ref{EQUATION_NEWTON_SECOND_LAW}) 
leads then to the linear harmonic oscillator equation with pure delay written as
\begin{equation}
	\ddot{x}(t) = M^{-1} K x(t - \tau) \text{ for } t \geq 0. \label{EQUATION_DELAY_OSCILLATOR}
\end{equation}
Problems similar to Equation (\ref{EQUATION_DELAY_OSCILLATOR}) also arise when modeling systems 
with distributed parameters such as general wave phenomena (cf. \cite{KhuPoAzi2013}).

Equations similar to (\ref{EQUATION_DELAY_OSCILLATOR}) are often referred to as delay or retarted differential equations.
After being transformed to a first order in time system on a Banach space $X$,
a general equation with constant delay can be written as
\begin{equation}
	\dot{u}(t) = H(t, u(t), u_{t}) \text{ for } t > 0, \quad
	u(0) = u^{0}, \quad u_{0} = \varphi.
	\label{EQUATION_DELAY_DIFFERENTIAL_EQUATION_GENERAL}
\end{equation}
Here, $\tau > 0$ is a fixed delay parameter,
$u_{t} := u(t + \cdot) \in L^{1}(-\tau, 0; X)$, $t \geq 0$, denotes the history variable,
$H$ is an $X$-valued operator defined on a subset of $[0, \infty) \times X \times L^{1}(-\tau, 0; X)$
and $u^{0} \in X$, $\varphi \in L^{1}(-\tau, 0; X)$ are appropriate initial data.
Equations of type (\ref{EQUATION_DELAY_DIFFERENTIAL_EQUATION_GENERAL}) have been intensively studied in the literature.
We refer the reader to the monographs by Els'gol'ts \& Norkin \cite{ElNo1973}
and Hale \& Lunel \cite{HaLu1993}
for a detailed treatment of Equations (\ref{EQUATION_DELAY_DIFFERENTIAL_EQUATION_GENERAL}) in finite-dimensional spaces $X$.
In contrast to this, results on Equation (\ref{EQUATION_DELAY_DIFFERENTIAL_EQUATION_GENERAL}) in infinite-dimensional spaces $X$ are less numerous.
A good overview can be found in the monograph of B\'{a}tkai \& Piazzera \cite{BaPia2005}.

Khusainov et al. considered in \cite{KhuAgDa1999} Equation (\ref{EQUATION_DELAY_DIFFERENTIAL_EQUATION_GENERAL}) in $\mathbb{R}^{n}$ with
\begin{align*}
	H(t, u(t), u_{t}) &= A_{1} u(t) + A_{2} u(t - \tau) \\
	&+ \big(u^{T}(t) \otimes b_{1}\big) u(t) + \big(u^{T}(t) \otimes b_{2}\big) u(t - \tau) + \big(u^{T}(t - \tau) \otimes b_{3}\big) u(t - \tau)\big)
\end{align*}
for symmetric matrices $A_{1}, A_{2} \in \mathbb{R}^{n \times n}$ and column vectors $b_{1}, b_{2}, b_{3} \in \mathbb{R}^{n}$
and proposed a rational Lyapunov function to study the asymptotic stability of solutions to this system.

In their work \cite{KhuAgKosKoj2000}, Khusainov, Agarwal et al. studied a modal, or spectrum, control problem for a linear delay equation on $\mathbb{R}^{n}$ reading as
\begin{equation}
	\dot{x}(t) = A x(t) + b u(t) \text{ for } t > 0
\end{equation}
with a feedback control $u(t) = \sum\limits_{j = 0}^{m} c_{j}^{T} x(t - j \tau)$ for some delay time $\tau > 0$ and parameter vectors $c_{j} \in \mathbb{R}^{n}$.
For canonical systems, they developed a method to compute the unknown parameters
such that the closed-loop system possesses the spectrum prescribed beforehand.
Under appropriate ``concordance'' conditions,
they were able to carry over their considerations for a rather broad class of non-canonical systems.

In the infite-dimensional situation, a rather general particular case of (\ref{EQUATION_DELAY_DIFFERENTIAL_EQUATION_GENERAL}) with $H(t, v, \psi) = A v + F(\psi)$
where $A$ generates a $C_{0}$-semigroup $(S(t))_{t \geq 0}$ on $X$ and $F$ is a nonlinear operator on $L^{2}(-\tau, 0; X)$
was studied by Travies \& Webb in their work \cite{TraWe1976}.
Under appropriate assumptions on $F$, they proved the integral equation corresponding to the weak formulation
of the delay equation given by
\begin{equation}
	u(t) = S(t) \varphi(0) + \int_{0}^{t} S(t - s) F(u_{s}) \mathrm{d}s \text{ for } t > 0
	\notag
\end{equation}
to possess a unique solution in $H^{1}_{\mathrm{loc}}(0, \infty; X)$.

Di Blasio et al. addressed in \cite{DiBlKuSi1983} a similar problem
\begin{equation}
	\dot{u}(t) = \big(A + B\big) u(t) L_{1} u(t - r) + L_{2} u_{t},
	\text{ for } t > 0, \quad u(0) = u^{0}, \quad u_{0} = \varphi
	\label{EQUATION_DELAY_DI_BLASIO}
\end{equation}
where $A$ generates a holomorphic $C_{0}$-semigroup on a Hilbert space $H$, $B$ is a perturbation of $A$
and $L_{1}, L_{2}$ are appropriate linear operators.
If $u^{0}$ and $\varphi$ possess a certain regularity,
they proved the existence of a unique strong solution in $H^{1}_{\mathrm{loc}}(0, \infty; X) \cap L^{2}_{\mathrm{loc}}\big(0, \infty; D(A)\big)$
by analyzing the $C_{0}$-semigroup inducing the the semiflow $t \mapsto (u(t), u_{t})$.
These results were elaborated on by Di Blasio et al. in \cite{DiBlKuSi1984}
leading to a generalization for the case of weighted and interpolation spaces and
including a desription of the associated infinitesimal generator.
Finally, the general $L^{p}$-case for $p \in (0, \infty)$ was investigated by Di Blasio in \cite{DiBl2003}.

Recently, in their work \cite{KhuPoRa2013}, Khusainov et al. proposed an explicit $L^{2}$-solution theory
for a non-homogeneous initial-boundary value problem for an isotropic heat equation with constant delay
\begin{equation}
	\begin{split}
		u_{t}(t, x) &=
		\partial_{i} \big(a_{ij}(x) \partial_{j} u(t, x)\big) + b_{i}(x) \partial_{i} u(t, x) + c(x) u(t, x) \\
		&+ \partial_{i} \big(\tilde{a}_{ij}(x) \partial_{j} u(t - \tau, x)\big) + \tilde{b}_{i}(x) \partial_{i} u(t - \tau, x) + \tilde{c}(x) u(t - \tau, x) + \\
		&+ f(t, x) \text{ for } (t, x) \in (0, \infty) \times \Omega, \\
		u(t, x) &= \gamma(t, x) \text{ for } (t, x) \in (0, \infty) \times \partial \Omega, \\
		u(0, x) &= u^{0}(x) \text{ for } x \in \Omega, \\
		u(t, x) &= \varphi(t, x) \text{ for } (t, x) \in (-\tau, 0) \times \Omega.
	\end{split}
	\notag
\end{equation}
where $\Omega \subset \mathbb{R}^{d}$ is a regular bounded domain
and the coefficient functions are appropriate.
Conditions assuring for exponential stability were also given.

Over the past decade, hyperbolic partial differential equations have attracted a considerable amound of attention, too.
In \cite{NiPi2006}, Nicaise \& Pignotti studied a homogeneous isotropic wave equation 
with an internal feedback with and without delay reading as
\begin{equation}
	\begin{split}
		\partial_{tt} u(t, x) - \triangle u(t, x) + a_{0} \partial_{t} u(t, x) + a \partial_{t} u(t - \tau, x) &= 0
		\text{ for } (t, x) \in (0, \infty) \times \Omega, \\
		u(t, x) &= 0 \text{ for } (t, x) \in (0, \infty) \times \Gamma_{0}, \\
		\frac{\partial u}{\partial \nu}(t, x) &= 0 \text{ for } (t, x) \in (0, \infty) \times \Gamma_{1}
	\end{split}
	\notag
\end{equation}
under usual initial conditions
where $\Gamma_{0}, \Gamma_{1} \subset \partial \Omega$ are relatively open in $\partial \Omega$
with $\bar{\Gamma}_{0} \cap \bar{\Gamma}_{1} = \emptyset$
and $\nu$ denotes the outer unit normal vector of a smooth bounded domain $\Omega \subset \mathbb{R}^{d}$.
They showed the problem to possess a unique global classical solution and proved the latter
to be exponentially stable if $a_{0} > a > 0$ or instable, otherwise.
These results have been carried over by Nicaise \& Pignotti \cite{NiPi2008}
and Nicaise et al. \cite{NiPiVa2011}
to the case time-varying internally distributed or boundary delays.

In \cite{KhuPoAzi2013}, Khusainov et al. considered a non-homogeneous initial-boundary value problem
for a one-dimensional wave equation with constant coefficients and a single constant delay
\begin{equation}
	\begin{split}
		\partial_{tt} u(t, x) &= a^{2} \partial_{xx} u(t - \tau, x) + b \partial_{x} u(t - \tau, x) + c u(t - \tau, x) \\
		&+ f(t, x) \text{ for } (t, x) \in (0, T) \times (0, l), \\
		u(t, x) &= \gamma(t, x) \text{ for } (t, x) \in (0, T) \times \{0, 1\}, \\
		u(0, x) &= u^{0}(x) \text{ for } x \in (0, 1), \\
		u(t, x) &= \varphi(t, x) \text{ for } t \in (-\tau, 0), x \in (0, 1).
	\end{split}
	\notag
\end{equation}
Under appropriate regularity and compatibility assumptions,
they proved the problem to possess a unique $C^{2}$-solution for any finite $T > 0$.
Their proof was based on extrapolation methods for $C_{0}$-semigroups and an explicit solution representation formula.

Recently, Khusainov \& Pokojovy presented in \cite{KhuPo2014} a Hilbert-space treatment of the initial-boundary value problem
for the equations of thermoelasticity with pure delay
\begin{equation}
	\begin{split}
		\partial_{tt} u(x, t) - a \partial_{xx} u(x, t - \tau) + b \partial_{x} \theta(x, t - \tau)  &= f(x, t) \text{ for } x \in \Omega, t > 0, \\
		\partial_{t} \theta(x, t) - c \partial_{xx} \theta(x, t - \tau) + d \partial_{tx} u(x, t - \tau) &= g(x, t) \text{ for } x \in \Omega, t > 0, \\
		u(0, t) = u(l, t) = 0, \quad \partial_{x} \theta(0, t) = \partial_{x} \theta(l, t) &= 0 \text{ for } t > 0, \\
		\phantom{\partial_{t}} u(x, 0) = u^{0}(x), \quad \phantom{\partial_{t}} u(x, t) &= u^{0}(x, t) \text{ for } x \in \Omega, t \in (-\tau, 0), \\
		\partial_{t} u(x, 0) = u^{1}(x), \quad \partial_{t} u(x, t) &= u^{1}(x, t) \text{ for } x \in \Omega, t \in (-\tau, 0), \\
		\phantom{\partial_{t}} \theta(x, 0) = \theta^{0}(x), \quad \phantom{\partial_{t}} \theta(x, t) &= \theta^{0}(x, t) \text{ for } x \in \Omega, t \in (-\tau, 0).
	\end{split}
	\notag
\end{equation}
Their proof exploited extrapolation techniques for strongly continuous semigroups
and an explicit solution representation formula.

In the present paper, we give a Banach space solution theory for Equation (\ref{EQUATION_DELAY_OSCILLATOR})
subject to appropriate initial conditions.
Our approach is solely based on the step method and does not incorporate any semigroup techniques.
In contrast to earlier works by Khusainov et al. \cite{KhuDiRuLu2008, KhuIvaKo2006, KhuPoAzi2013},
we only require the invertibility and not the positivity of $M^{-1} K$ in Equation (\ref{EQUATION_DELAY_OSCILLATOR}).

In Section \ref{SECTION_CLASSICAL_HARMONIC_OSCILLATOR},
we briefly outline some seminal results on second-order abstract Cauchy problems.
In our main Section \ref{SECTION_CLASSICAL_HARMONIC_OSCILLATOR_WITH_PURE_DELAY},
we prove the existence and uniqueness of solutions to the Cauchy problem for the delay equation (\ref{EQUATION_DELAY_OSCILLATOR})
as well as their continuous dependence on the data.
Next, we give an explicit solution representation formula in a closed form
based on the delayed exponential function introduced by Khusainov \& Shuklin in \cite{KhuShu2005}.
Finally, we prove the solution of the delay equation
to converge to the solution of the original second order abstract differential equation
as the delay parameter $\tau$ goes to zero.

\section{Classical harmonic oscillator}
	\label{SECTION_CLASSICAL_HARMONIC_OSCILLATOR}
	For the sake of completeness, 
	we briefly discuss the initial value problem for the harmonic oscillator being a second order in time abstact differential equation
	\begin{equation}
		\ddot{x}(t) - \Omega^{2} x(t) = f(t) \text{ for } t \geq 0
		\label{EQUATION_HARMONIC_OSCILLATOR_GENERAL}
	\end{equation}
	subject to the initial conditions
	\begin{equation}
		x(0) = x_{0} \in D(\Omega), \quad \dot{x}(0) = x_{1} \in X.
		\label{EQUATION_HARMONIC_OSCILLATOR_GENERAL_IC}
	\end{equation}
	Here, we assume the linear operator $\Omega \colon D(\Omega) \subset X \to X$
	to be continuously invertible and generate a $C_{0}$-group $(e^{t\Omega})_{t \in \mathbb{R}} \subset L(X)$ 
	on a (real or complex) Banach space $X$ with
	$L(X)$ denoting the space of bounded, linear operators on $X$ equipped with the norm
	$\|A\|_{L(X)} := \sup\big\{\|Ax\|_{X} \;|\; x \in X, \|x\|_{X} \leq 1\big\}$.
	A more rigorous treatment of this problem can be found in \cite[Section 3.14]{ArBaHieNeu2001}.
	
	The general solution to the homogeneous equation is known to read as
	\begin{equation}
		x_{h}(t) = e^{\Omega t} c_{1} + e^{-\Omega t} c_{2} \text{ for } t \geq 0 \notag
	\end{equation}
	with some $c_{1}, c_{2} \in D(\Omega)$.
	Vectors $c_{1}, c_{2}$ can be computed using the initial conditions from Equation (\ref{EQUATION_HARMONIC_OSCILLATOR_GENERAL_IC}) 
	leading to a system of linear operator equations
	\begin{equation}
		c_{1} + c_{2} = x_{0}, \quad \Omega c_{1} - \Omega c_{2} = x_{1}. \notag
	\end{equation}
	The latter is uniquely solved by
	\begin{equation}
		c_{1} = \tfrac{1}{2} \Omega^{-1} (\Omega x_{0} + x_{1}), \quad c_{1} = \tfrac{1}{2} \Omega^{-1} (\Omega x_{0} - x_{1}). \notag
	\end{equation}
	Thus, the unique solution of the homogeneous equation with the initial conditions (\ref{EQUATION_HARMONIC_OSCILLATOR_GENERAL_IC}) is given by
	\begin{equation}
		x_{h}(t) = \tfrac{1}{2} \Omega^{-1} e^{\Omega t} (\Omega x_{0} + x_{1}) + \tfrac{1}{2} \Omega^{-1} e^{-\Omega t} (\Omega x_{0} - x_{1})
		\text{ for } t \geq 0
	\end{equation}
	or, equivalently,
	\begin{equation}
		x_{h}(t) = \tfrac{1}{2} (e^{\Omega t} + e^{-\Omega t}) x_{0} + \tfrac{1}{2} \Omega^{-1} (e^{\Omega t} - e^{-\Omega t}) x_{1}
		\text{ for } t \geq 0.
		\label{EQUATION_HARMONIC_OSCILLATOR_GENERAL_SOLUTION_HOMOGENEOUS_EQUATION}
	\end{equation}
	A particular solution to the non-homogeneous equation with zero initial conditions will be determined in the Cauchy form
	\begin{equation}
		x_{p}(t) = \int_{0}^{t} K(t, s) f(s) \mathrm{d}s \text{ for } t \geq 0.
		\label{EQUATION_HARMONIC_OSCILLATOR_GENERAL_SOLUTION_NONHOMOGENEOUS_EQUATION_ANSATZ}
	\end{equation}
	We refer the reader to \cite[Chapter 1]{ArBaHieNeu2001} for the definition of Bochner integrals for $X$-valued functions.
	In Equation (\ref{EQUATION_HARMONIC_OSCILLATOR_GENERAL_SOLUTION_NONHOMOGENEOUS_EQUATION_ANSATZ}),
	the function $K \in C^{0}([0, \infty) \times [0, \infty), L(X))$ is the Cauchy kernel, i.e., 
	for any fixed $s \geq 0$, the function $K(\cdot, s)$ is the solution of the homogeneous problem satisfying the initial conditions
	\begin{equation}
		K(t, s)\big|_{t = s} = 0_{L(X)}, \quad
		\partial_{t} K(t, s)\big|_{t = s} = \mathrm{id}_{X}. \notag
	\end{equation}
	Using the ansatz
	\begin{equation}
		K(t, s) = e^{\Omega t} c_{1}(s) + e^{-\Omega t} c_{2}(s) \text{ for } t, s \geq 0 \notag
	\end{equation}
	for some $c_{1}, c_{2} \in C^{1}([0, \infty), L(X))$
	and taking into account the initial conditions, we arrive at
	\begin{equation}
		K(t, s)\big|_{t = s} = e^{\Omega t} c_{1}(s) + e^{-\Omega t} c_{2}(s) = 0_{L(X)}, \quad
		\partial_{t} K(t, s)\big|_{t = s} = \Omega e^{\Omega s} c_{1}(s) - \Omega e^{-\Omega s} c_{2}(s) = \mathrm{id}_{X}. \notag
	\end{equation}
	Solving this system with generalized Cramer's rule, we obtain for $s \geq 0$
	\begin{equation}
		\begin{split}
			c_{1}(s) &=
			\left(\det{}_{L(X)}\begin{pmatrix}
				e^{\Omega s} & e^{-\Omega s} \\
				\Omega e^{\Omega s} & -\Omega e^{-\Omega s}
			\end{pmatrix}\right)^{-1}
			\det{}_{L(X)}\begin{pmatrix}
				0_{L(X)} & e^{-\Omega s} \\
				\mathrm{id}_{X} & -\Omega e^{-\Omega s}
			\end{pmatrix}
			= \tfrac{1}{2} \Omega^{-1} e^{-\Omega s}, \\
			c_{2}(s) &=
			\left(\det{}_{L(X)}\begin{pmatrix}
				e^{\Omega s} & e^{-\Omega s} \\
				\Omega e^{\Omega s} & -\Omega e^{-\Omega s}
			\end{pmatrix}\right)^{-1}
			\det{}_{L(X)}\begin{pmatrix}
				e^{\Omega s} & 0_{L(X)} \\
				\Omega e^{\Omega s} & \mathrm{id}_{X}
			\end{pmatrix}
			= \tfrac{1}{2} \Omega^{-1} e^{-\Omega s}.
		\end{split}
		\notag
	\end{equation}
	Thus, the Cauchy kernel is given by
	\begin{equation}
		K(t, s) = \tfrac{1}{2} \Omega^{-1} (e^{\Omega(t - s)} - e^{-\Omega(t - s)}) \text{ for } t, s \geq 0, \notag
	\end{equation}
	whereas the particular solution satisfying zero initial conditions reads as
	\begin{equation}
		x_{p}(t) = \frac{1}{2} \Omega^{-1} \int_{0}^{t} (e^{\Omega (t - s)} - e^{-\Omega (t - s)}) f(s) \mathrm{d}s \text{ for } t \geq 0. \notag
	\end{equation}
	
	Hence, for $x_{0} \in D(\Omega)$, $x_{1} \in X$ and $f \in L^{1}_{\mathrm{loc}}(0, \infty; X)$,
	the unique mild solution $x \in W^{1, 1}_{\mathrm{loc}}(0, \infty; X)$
	to the Cauchy problem (\ref{EQUATION_HARMONIC_OSCILLATOR_GENERAL})--(\ref{EQUATION_HARMONIC_OSCILLATOR_GENERAL_IC}) 
	can be written as
	\begin{equation}
		\begin{split}
			x(t) &= \tfrac{1}{2} (e^{\Omega t} + e^{-\Omega t}) x_{0} +
			\tfrac{1}{2} \Omega^{-1} (e^{\Omega t} - e^{-\Omega t}) x_{1} \\
			&+ \tfrac{1}{2} \Omega^{-1} \int_{0}^{t} (e^{\Omega (t - s)} - e^{-\Omega (t - s)}) f(s) \mathrm{d}s
			\text{ for } t \geq 0.
		\end{split}
		\label{EQUATION_HARMONIC_OSCILLATOR_GENERAL_EXPLICIT_SOLUTION}
	\end{equation}
	If the data additionally satisfy $x_{0} \in D(\Omega^{2})$, $x_{1} \in D(\Omega)$ and 
	$f \in W^{1, 1}_{\mathrm{loc}}(0, \infty; X) \cup C^{0}\big([0, \infty), D(\Omega^{2})\big)$,
	then the mild solution $x$ given in Equation (\ref{EQUATION_HARMONIC_OSCILLATOR_GENERAL_EXPLICIT_SOLUTION})
	is a classical solution satisfying
	$x \in C^{2}\big([0, \infty), X\big) \cap C^{1}\big([0, \infty), D(\Omega)\big) \cap C^{0}\big([0, \infty), D(\Omega^{2})\big)$.

\section{The linear oscillator with pure delay}
	\label{SECTION_CLASSICAL_HARMONIC_OSCILLATOR_WITH_PURE_DELAY}
	In this section, we consider a Cauchy problem for the linear oscillator with a single pure delay
	\begin{equation}
		\ddot{x}(t) - \Omega^{2} x(t - 2\tau) = f(t) \text{ for } t \geq 0
		\label{EQUATION_LINEAR_OSCILLATOR_WITH_DELAY_GENERAL}
	\end{equation}
	subject to the initial condition
	\begin{equation}
		x(t) = \varphi(t) \text{ for } t \in [-2\tau, 0].
		\label{EQUATION_LINEAR_OSCILLATOR_WITH_DELAY_GENERAL_IC}
	\end{equation}
	Here, $X$ is a Banach space, $\Omega \in L(X)$ is a bounded, linear operator and 
	$\varphi \in C^{1}\big([-2\tau, 0], X\big)$, $f \in L^{1}_{\mathrm{loc}}(0, \infty; X)$ are given functions.
	In contrast to Section \ref{SECTION_CLASSICAL_HARMONIC_OSCILLATOR},
	the boundedness of $\Omega$ is indespensable here.
	Indeed, Dreher et al. proved in \cite{DreQuiRa2009}
	that Equations (\ref{EQUATION_LINEAR_OSCILLATOR_WITH_DELAY_GENERAL})--(\ref{EQUATION_LINEAR_OSCILLATOR_WITH_DELAY_GENERAL_IC})
	are ill-posed even if $X$ is a Hilbert space and $\Omega$ possesses a sequence of eigenvalues $(\lambda_{n})_{n \in \mathbb{N}} \subset \mathbb{R}$
	with $\lambda_{n} \to \infty$ or $\lambda_{n} \to -\infty$ as $n \to \infty$.
	The necessity for $\Omega$ being bounded has also been pointed out by Rodrigues et al. in \cite{RoWu2007}
	when treating a linear heat equation with pure delay.

	\begin{definition}
		A function $x \in C^{1}\big([-2\tau, \infty), X\big) \cap C^{2}\big([-2\tau, 0], X\big) \cap C^{2}\big([0, \infty), X\big)$
		satisfying Equations (\ref{EQUATION_LINEAR_OSCILLATOR_WITH_DELAY_GENERAL})--(\ref{EQUATION_LINEAR_OSCILLATOR_WITH_DELAY_GENERAL_IC})
		pointwise is called a classical solution to the Cauchy problem (\ref{EQUATION_LINEAR_OSCILLATOR_WITH_DELAY_GENERAL})--(\ref{EQUATION_LINEAR_OSCILLATOR_WITH_DELAY_GENERAL_IC}).
	\end{definition}

	A mild formulation of (\ref{EQUATION_LINEAR_OSCILLATOR_WITH_DELAY_GENERAL})--(\ref{EQUATION_LINEAR_OSCILLATOR_WITH_DELAY_GENERAL_IC}) is given by
	\begin{align}
		\dot{x}(t) &= \dot{x}(0) + \Omega^{2} \int_{0}^{t} x(s - 2\tau) \mathrm{d}s + \int_{0}^{t} f(s) \mathrm{d}s \text{ for } t \geq 0, \label{EQUATION_LINEAR_OSCILLATOR_WITH_DELAY_MILD_FORMULATION} \\
		x(t) &= \varphi(t) \text{ for } t \in [-2\tau, 0]. \label{EQUATION_LINEAR_OSCILLATOR_WITH_DELAY_MILD_FORMULATION_IC}
	\end{align}
	
	\begin{definition}
		A function $x \in C^{1}\big([-2\tau, \infty), X\big)$ satisfying Equations (\ref{EQUATION_LINEAR_OSCILLATOR_WITH_DELAY_MILD_FORMULATION})--(\ref{EQUATION_LINEAR_OSCILLATOR_WITH_DELAY_MILD_FORMULATION_IC})
		is called a mild solution to the Cauchy problem (\ref{EQUATION_LINEAR_OSCILLATOR_WITH_DELAY_GENERAL})--(\ref{EQUATION_LINEAR_OSCILLATOR_WITH_DELAY_GENERAL_IC}).
	\end{definition}

	By the virtue of fundamental theorem of calculus,
	any mild solution $x$ to (\ref{EQUATION_LINEAR_OSCILLATOR_WITH_DELAY_GENERAL})--(\ref{EQUATION_LINEAR_OSCILLATOR_WITH_DELAY_GENERAL_IC})
	with $x \in C^{1}\big([-2\tau, \infty), X\big) \cap C^{2}\big([-2\tau, 0], X\big) \cap C^{2}\big([0, \infty), X\big)$
	is also a classical solution.
	Obviously, for the problem (\ref{EQUATION_LINEAR_OSCILLATOR_WITH_DELAY_GENERAL})--(\ref{EQUATION_LINEAR_OSCILLATOR_WITH_DELAY_GENERAL_IC})
	to possess a classical solution, one necessarily requires $\varphi \in C^{2}\big([-2\tau, 0], X\big)$.

	In the following subsection,
	we want to study the existence and uniquess of mild and classical solutions to the Cauchy problem (\ref{EQUATION_LINEAR_OSCILLATOR_WITH_DELAY_GENERAL})--(\ref{EQUATION_LINEAR_OSCILLATOR_WITH_DELAY_GENERAL_IC})
	as well as their continuous dependence on the data.

	\subsection{Existence and uniqueness}
	Rather then using the semigroup approach (cf. \cite[Chapter 2]{HaLu1993}),
	we decided to use the more straightforward step method here
	reducing (\ref{EQUATION_LINEAR_OSCILLATOR_WITH_DELAY_MILD_FORMULATION})--(\ref{EQUATION_LINEAR_OSCILLATOR_WITH_DELAY_MILD_FORMULATION_IC})
	to a difference equation on the functional vector space $\hat{C}^{1}_{2\tau}(\mathbb{N}_{0}, X)$ defined as follows.
	\begin{definition}
		Let $X$ be a Banach space, $\tau > 0$ and $s \in \mathbb{N}_{0}$.
		We introduce the metric vector space
		\begin{equation}
			\begin{split}
				\hat{C}^{s}_{\tau}(\mathbb{N}_{0}, X) &:= l^{\infty}_{\mathrm{loc}}\big(\mathbb{N}_{0}, C^{s}\big([-\tau, 0], X)\big)\big) \\
				&:= \Big\{x = (x_{n})_{n \in \mathbb{N}_{0}} \,\big|\, x_{n} \in C^{s}\big([-\tau, 0], X\big) \text{ for } n \in \mathbb{N}_{0}, \\
				&\phantom{:= \Big\{x = (x_{n})_{n \in \mathbb{N}_{0}} \,\big|\,} \frac{\mathrm{d}^{j}}{\mathrm{d}t^{j}} x_{n}(-\tau) = \frac{\mathrm{d}^{j}}{\mathrm{d}t^{j}} x_{n-1}(0) \text{ for } j = 0, \dots, s - 1, n \in \mathbb{N}\Big\}
			\end{split}
			\notag
		\end{equation}
		equipped with the distance function
		\begin{equation}
			d_{\hat{C}_{\tau}^{s}(\mathbb{N}_{0}, X)}(x, y) :=
			\sum_{n \in N} 2^{-n}
			\frac{\max\limits^{}_{k = 0, \dots, n} \|x_{k} - y_{k}\|_{C^{s}([-\tau, 0], X)}}{1 + \max\limits^{}_{k = 0, \dots, n} \|x_{k} - y_{k}\|_{C^{s}([-\tau, 0], X)}}
			\text{ for } x, y \in \hat{C}_{\tau}^{s}(\mathbb{N}_{0}, X). \notag
		\end{equation}
	\end{definition}
	Obviously, $\hat{C}^{s}_{\tau}(\mathbb{N}_{0}, X)$ is a complete metric space which is isometrically isomorphic 
	to the metric space $C^{s}_{\tau}\big([-\tau, \infty), X\big) := C^{s}\big([-\tau, \infty), X\big)$ equipped with the distance
	\begin{equation}
		d_{C^{s}_{\tau}([0, \infty), X)}(x, y) :=
		\sum_{n \in N} 2^{-n}
		\frac{\|x - y\|_{C^{s}([-\tau, \tau n], X)}}
		{1 + \|x - y\|_{C^{s}([-\tau, \tau n], X)}} \text{ for } x, y \in C^{s}\big([-\tau, \infty), X\big). \notag
	\end{equation}

	For any $x \colon [-\tau, \infty) \to X$,
	we define for $n \in \mathbb{N}_{0}$ the $n$-th segment of $x$ by means of
	\begin{equation}
		x_{n} := x(n\tau + s) \text{ for } s \in [-\tau, 0]. \notag
	\end{equation}
	By induction, $x$ is a mild solution of (\ref{EQUATION_LINEAR_OSCILLATOR_WITH_DELAY_GENERAL})--(\ref{EQUATION_LINEAR_OSCILLATOR_WITH_DELAY_GENERAL_IC})
	if and only if $(x_{n})_{n \in \mathbb{N}_{0}} \in \hat{C}^{1}_{2\tau}(\mathbb{N}_{0}, X)$ solves
	\begin{equation}
		\begin{split}
			\dot{x}_{n}(s) &= \dot{x}_{n-1}(0) + \Omega^{2} x_{n-1}(s) + \int_{2(n-1)\tau}^{2(n-1)\tau + s} f(\sigma) \mathrm{d}\sigma \text{ for } s \in [-2\tau, 0], n \in \mathbb{N}, \\
			x_{0}(s) &= \varphi(s) \text{ for } s \in [-2\tau, 0].
		\end{split}
		\label{EQUATION_DIFFERENCE_EQUATION_FOR_DOT_X_N}
	\end{equation}

	\begin{theorem}
		\label{THEOREM_MILD_SOLUTION_DISCRETE}
		Equation (\ref{EQUATION_DIFFERENCE_EQUATION_FOR_DOT_X_N}) has a unique solution $(x_{n})_{n \in \mathbb{N}_{0}} \in \hat{C}^{1}_{2\tau}(\mathbb{N}_{0}, X)$.
		Moreover, $x$ continuously depends on the data in sense of the estimate
		\begin{equation}
			\|x_{n}\|_{C^{1}([-2\tau, 0], X)} \leq
			\kappa^{n} \Big(\|\varphi\|_{C^{1}([-2\tau, 0], X)} + \|f\|_{L^{1}(0, 2\tau n;  X)}\Big)
			\text{ for any } n \in \mathbb{N} \notag
		\end{equation}
		with $\kappa := 1 + (1 + 2\tau) \big(1 + \|\Omega\|_{L(X)}^{2}\big)$.
	\end{theorem}

	\begin{proof}
		By the virtue of the fundamental theorem of calculus, Equation (\ref{EQUATION_DIFFERENCE_EQUATION_FOR_DOT_X_N}) is satisfied if and only if
		\begin{align}
			x_{n}(s) &= x_{n-1}(0) + (s - 2\tau) \dot{x}_{n-1}(0)
			+ \Omega^{2} \int_{-2\tau}^{s} x_{n-1}(\sigma) \mathrm{d}\sigma \\
			&+ \int_{-2\tau}^{s} \int_{2(n-1)\tau}^{2(n-1)\tau + \sigma} f(\xi) \mathrm{d}\xi \mathrm{d}\sigma \text{ for } s \in [-2\tau, 0], n \in \mathbb{N}, \label{EQUATION_DIFFERENCE_EQUATION_1} \\
			x_{n}(-2\tau) &= x_{n-1}(0), \quad \dot{x}_{n}(-2\tau) = \dot{x}_{n-1}(0) \text{ for } n \in \mathbb{N}, \label{EQUATION_DIFFERENCE_EQUATION_2} \\
			x_{0}(s) &= \varphi(s) \text{ for } s \in [-2\tau, 0]. \label{EQUATION_DIFFERENCE_EQUATION_3}
		\end{align}
		By induction, we can easily show that for any $n \in \mathbb{N}$
		there exists a unique local solution $(x_{0}, x_{1}, \dots, x_{n}) \in \Big(C^{1}\big([-2\tau, 0], X\big)\Big)^{n + 1}$
		to (\ref{EQUATION_DIFFERENCE_EQUATION_1})--(\ref{EQUATION_DIFFERENCE_EQUATION_3}) up to the index $n$.
		Here, we used the Sobolev embedding theorem stating
		\begin{equation}
			W^{1, 1}(0, T; X) \hookrightarrow C^{0}\big([0, T], X\big) \text{ for any } T > 0. \notag
		\end{equation}
		Further, we can estimate
		\begin{equation}
			\begin{split}
				\|x_{n}\|_{C^{0}([-2\tau, 0], X)} &\leq
				\Big(1 + 2\tau \big(1 + \|\Omega\|_{L(X)}^{2}\big)\Big) \|x_{n-1}\|_{C^{1}([-2\tau, 0], X)} \\
				&+ 2 \tau \|f\|_{L^{1}(2(n-1)\tau, 2n\tau; X)}.
			\end{split}
			\label{EQUATION_ESIMATE_FOR_X_N}
		\end{equation}
		Similarly, Equation (\ref{EQUATION_DIFFERENCE_EQUATION_FOR_DOT_X_N}) yields
		\begin{equation}
			\|\dot{x}_{n}\|_{C^{0}([-2\tau, 0], X)} \leq
			\big(1 + \|\Omega\|_{L(X)}^{2}\big) \|x_{n-1}\|_{C^{0}([-2\tau, 0], X)} + \|f\|_{L^{1}(2(n-1)\tau, 2n\tau; X)}.
			\label{EQUATION_ESIMATE_FOR_DOT_X_N}
		\end{equation}
		Equations (\ref{EQUATION_ESIMATE_FOR_X_N}) and (\ref{EQUATION_ESIMATE_FOR_DOT_X_N}) imply together
		\begin{align*}
			\|x_{n}\|_{C^{1}([-2\tau, 0], X)} &\leq
			\kappa \big(\|\varphi\|_{C^{1}([-2\tau, 0], X)} + \|f\|_{L^{1}(2(n-1)\tau, n\tau: X)}\big).
		\end{align*}
		By induction, we then get for any $n \in \mathbb{N}$
		\begin{equation}
			\|x_{n}\|_{C^{1}([-2\tau, 0], X)} \leq
			\kappa^{n} \big(\|\varphi\|_{C^{0}([-2\tau, 0], X)} + \|f\|_{L^{1}(0, 2\tau n, X)}\big) \notag
		\end{equation}
		which finishes the proof.
	\end{proof}

	Letting $x(t) := x_{k}(t - 2 (k + 1) \tau)$ for $t \geq 0$ and $k := \lfloor \tfrac{t}{2\tau}\rfloor \in \mathbb{N}_{0}$,
	we obtain the unique mild solution $x$ of Equations (\ref{EQUATION_LINEAR_OSCILLATOR_WITH_DELAY_GENERAL})--(\ref{EQUATION_LINEAR_OSCILLATOR_WITH_DELAY_GENERAL_IC}).
	\begin{corollary}
		\label{COROLLARY_MILD_SOLUTION}
		Equations (\ref{EQUATION_LINEAR_OSCILLATOR_WITH_DELAY_GENERAL})--(\ref{EQUATION_LINEAR_OSCILLATOR_WITH_DELAY_GENERAL_IC})
		possess a unique mild solution $x$ satisfying for any $T := 2n\tau$, $n \in \mathbb{N}$,
		\begin{equation}
			\|x\|_{C^{1}([-2\tau, T], X)} \leq
			\kappa^{n} \Big(\|\varphi\|_{C^{1}([-2\tau, T], X)} + \|f\|_{L^{1}(0, T;  X)}\Big)
			\text{ for any } n \in \mathbb{N}. \notag
		\end{equation}
		with $\kappa := 1 + (1 + 2\tau) \big(1 + \|\Omega\|_{L(X)}^{2}\big)$.
	\end{corollary}

	\begin{theorem}
	      Under additional conditions $\varphi \in C^{2}\big([-2\tau, 0], X\big)$ and
	      $f \in C^{0}\big([0, \infty), X\big)$, the unique mild solution given in Corollary \ref{COROLLARY_MILD_SOLUTION}
	      is a classical solution.
	\end{theorem}

	\begin{proof}
		Differentiating Equation (\ref{EQUATION_DIFFERENCE_EQUATION_FOR_DOT_X_N}) with respect to $t$,
		using the assumptions and the fact that $x \in C^{1}\big([-2\tau, \infty), X\big)$,
		we deduce that $x|_{[-2\tau, 0]} \equiv \varphi \in C^{2}\big([-2\tau, 0], X\big)$ and
		\begin{equation}
		      \ddot{x} = \Omega^{2} x(\cdot - 2\tau) + f \in C^{0}\big([0, \infty), X\big). \notag
		\end{equation}
		Hence, $x \in C^{1}\big([-2\tau, \infty), X\big) \cap C^{2}\big([-2\tau, 0], X\big) \cap C^{2}\big([0, \infty), X\big)$
		and is thus a classical solution of Equations
		(\ref{EQUATION_LINEAR_OSCILLATOR_WITH_DELAY_GENERAL})--(\ref{EQUATION_LINEAR_OSCILLATOR_WITH_DELAY_GENERAL_IC}).
	\end{proof}

	\subsection{Explicit representation of solutions}
	\label{SECTION_HARMONIC_OSCILLATOR_WITH_DELAY_SOLUTION_REPRESENTATION} 
	Following Khusainov \& Shuklin \cite{KhuShu2005} and Khusainov et al. \cite{KhuPoRa2013}, 
	we define for $t \in \mathbb{R}$ the operator-valued delayed exponential function
	\begin{equation}
		\exp_{\tau}(t; \Omega) :=
		\left\{
		\begin{array}{cc}
			0_{L(X)}, & -\infty < t < -\tau, \\
			\mathrm{id}_{X}, & -\tau \leq t < 0, \\
			\mathrm{id}_{X} + \Omega \frac{t}{1!}, & 0 \leq t < \tau, \\
			\mathrm{id}_{X} + \Omega \frac{t}{1!} + \Omega^{2} \frac{(t - \tau)^{2}}{2!}, & \tau \leq t < 2\tau, \\
			\dots & \dots \\
			\mathrm{id}_{X} + \Omega \frac{t}{1!} + \Omega^{2} \frac{(t - \tau)^{2}}{2!} + \dots +
			\Omega^{k} \frac{\left(t - (k - 1) \tau\right)^{k}}{k!}, & (k - 1) \tau \leq t < k\tau, \\
			\dots & \dots.
		\end{array}
		\right.
		\label{EQUATION_DEFINITION_OF_DELAYED_EXPONENTIAL}
	\end{equation}

	Throughout this Section, we additionally assume that
	$\Omega \colon X \to X$ is an isomorphism from the Banach space $X$ onto itself.
	\begin{theorem}
		The delayed exponential function $\exp_{\tau}(\cdot; \Omega)$ lies in 
		$C^{0}\big([-\tau, \infty), X\big) \cap C^{1}\big([0, \infty), X\big) \cap C^{2}\big([\tau, \infty), X\big)$
		and solves the Cauchy problem
		\begin{align}
			\ddot{x}(t) - \Omega^{2} x(t - 2\tau) &= 0_{X} \text{ for } t \geq \tau, \label{EQUATION_DDE_FOR_DEXP_1} \\
			x(t) &= \varphi(t) \text{ for } t \in [-\tau, \tau] \label{EQUATION_DDE_FOR_DEXP_2}
		\end{align}
		where
		\begin{equation}
			\varphi(t) = \left\{
			\begin{array}{cc}
				\mathrm{id}_{X}, & -\tau \leq t < 0, \\
				\mathrm{id}_{X} + \Omega t, & 0 \leq t \leq \tau. 
			\end{array}
			\right. \notag
		\end{equation}
	\end{theorem}

	\begin{proof}
		To prove the smoothness of $x$, we first note that
		$x$ is an operator-valued polynomial and thus analytic on each of the intervals $[(k - 1) \tau, k \tau]$ for $k \in \mathbb{Z}$.
		By the definition of $\exp_{\tau}(\cdot; \Omega)$, we further find
		\begin{equation}
			\frac{\mathrm{d}^{j}}{\mathrm{d} t^{j}} x(k\tau - 0) = \frac{\mathrm{d}^{j}}{\mathrm{d} t^{j}} x(k\tau + 0) \text{ for } j = 0, \dots, k, \quad k \in \mathbb{N}_{0}. \notag
		\end{equation}
		Hence, $x \in C^{0}\big([-\tau, \infty), X\big) \cap C^{1}\big([0, \infty), X\big) \cap C^{2}\big([\tau, \infty), X\big)$.

		For $k \in \mathbb{N}$, $k \geq 2$, we have
		\begin{equation}
			x(t) = \mathrm{id}_{X} + \Omega \frac{t}{1!} + \Omega^{2} \frac{(t - \tau)^{2}}{2!} + \Omega^{3} \frac{(t - 3\tau)^{3}}{4!} +
			\Omega^{4} \frac{(t - 3\tau)^{4}}{4!} + \dots +
			\Omega^{k} \frac{\left(t - (k - 1) \tau\right)^{k}}{k!}. \notag
		\end{equation}
		For $t \geq \tau$, differentiation yields
		\begin{align*}
			\dot{x}(t) &=
			\Omega + \Omega^{2} \frac{t - \tau}{2!} + \Omega^{3} \frac{(t - 2\tau)^{2}}{4!} + \Omega^{4} \frac{(t - 3\tau)^{3}}{3!} + \dots +
			\Omega^{k} \frac{\left(t - (k - 1) \tau\right)^{k-1}}{(k - 1)!} \\
			&= \Omega \Big(
			\mathrm{id}_{X} + \Omega \frac{t - \tau}{2!} + \Omega^{2} \frac{(t - 2\tau)^{2}}{4!} + \Omega^{3} \frac{(t - 3\tau)^{3}}{3!} + \dots +
			\Omega^{k-1} \frac{\left(t - (k - 1) \tau\right)^{k-1}}{(k - 1)!}\Big) \\
			&= \Omega \exp_{\tau}(t - \tau; \Omega) = \Omega x(t - \tau) \notag
		\end{align*}
		and therefore
		\begin{align*}
			\ddot{x}(t) &= \Omega^{2} +
			\Omega^{3} \frac{t - 2\tau}{1!} + \Omega^{4} \frac{(t - 3\tau)^{2}}{2!} + \dots +
			\Omega^{k} \frac{\left(t - (k - 1) \tau\right)^{k-2}}{(k - 2)!} \\
			&= \Omega^{2} \Big(
			\mathrm{id}_{X} + \Omega \frac{t - 2\tau}{1!} + \Omega^{2} \frac{(t - 3\tau)^{2}}{2!} + \dots +
			\Omega^{k-2} \frac{\left(t - (k - 1) \tau\right)^{k-2}}{(k - 2)!}\Big) \\
			&= \Omega^{2} \exp_{\tau}(t - 2\tau; \Omega) = \Omega^{2} x(t - 2\tau).
		\end{align*}
		Hence, $x$ satisfies Equation (\ref{EQUATION_DDE_FOR_DEXP_1}).
		Finally, by the definition of $\exp_{\tau}(\cdot; \Omega)$,
		$x$ satisfies Equation (\ref{EQUATION_DDE_FOR_DEXP_2}), too.
	\end{proof}

	\begin{corollary}
		The delayed exponential function $\exp_{\tau}(\cdot; -\Omega)$ lies in $C^{0}([-\tau, \infty), X) \cap C^{1}\big([0, \infty), X\big) \cap C^{2}\big([\tau, \infty), X) \cap $
		and solves the Cauchy problem (\ref{EQUATION_DDE_FOR_DEXP_1})--(\ref{EQUATION_DDE_FOR_DEXP_2})
		with the initial data
		\begin{equation}
			\varphi(t) = \left\{
			\begin{array}{cc}
				\mathrm{id}_{X}, & -\tau \leq t < 0, \\
				\mathrm{id}_{X} - \Omega t, & 0 \leq t \leq \tau. 
			\end{array}
			\right. \notag
		\end{equation}
	\end{corollary}

	We define the functions
	\begin{equation}
		\begin{split}
			x_{1}(t; \Omega) &:= \frac{1}{2} \big(\exp_{\tau}(t; \Omega) + \exp_{\tau}(t; -\Omega)\big) \text{ for } t \geq -\tau, \\
			x_{2}(t; \Omega) &:= \frac{1}{2} \Omega^{-1} \big(\exp_{\tau}(t; \Omega) - \exp_{\tau}(t; -\Omega)\big) \text{ for } t \geq -\tau. 
		\end{split}
		\label{EQUATION_FUNCTIONS_X_1_AND_X_2}
	\end{equation}
	From Equation (\ref{EQUATION_DEFINITION_OF_DELAYED_EXPONENTIAL}), we explicitly obtain
	\begin{equation}
		x_{1}(t; \Omega) = \left\{
		\begin{array}{cc}
			\mathrm{id}_{X}, & -\tau \leq t < \tau, \\
			\mathrm{id}_{X} + \Omega^{2} \frac{(t - \tau)^{2}}{2!}, & \tau \leq t < 3\tau, \\
			\mathrm{id}_{X} + \Omega^{2} \frac{(t - \tau)^{2}}{2!} + \Omega^{4} \frac{(t - 3\tau)^{4}}{4!}, & 3\tau \leq t < 5\tau, \\
			\dots & \dots \\
			\mathrm{id}_{X} + \Omega^{2} \frac{(t - \tau)^{2}}{2!} + \dots +
			\Omega^{2k} \frac{(t - (2k - 1) \tau)^{2k}}{(2k)!}, & (2k - 1) \tau \leq t < (2k + 1) \tau
		\end{array}\right.
		\notag
	\end{equation}
	and
	\begin{equation}
		x_{2}(t; \Omega) = \left\{
		\begin{array}{cc}
			0_{L(X)}, & -\tau \leq t < 0, \\
			\mathrm{id}_{X} \frac{t}{1!}, & 0 \leq t < 2\tau, \\
			\mathrm{id}_{X} \frac{t}{1!} + \Omega^{2} \frac{(t - 2\tau)^{3}}{3!}, & 2\tau \leq t < 4\tau, \\
			\mathrm{id}_{X} \frac{t}{1!} + \Omega^{2} \frac{(t - 2\tau)^{3}}{3!} +
			\Omega^{4} \frac{(t - 4\tau)^{5}}{5!}, & 4\tau \leq t < 6\tau, \\
			\dots & \dots \\
			\mathrm{id}_{X} \frac{t}{1!} + \Omega^{2} \frac{(t - 2\tau)^{3}}{3!} + \dots +
			\Omega^{2k} \frac{(t - (2k) \tau)^{2k+1}}{(2k + 1)!}, & 2k \tau \leq t < 2(k + 1) \tau.
		\end{array}\right.
		\notag
	\end{equation}
	Obviously, $x_{1}$ and $x_{2}$ are even functions with respect to $\Omega$.
	Figure \ref{FIGURE_FUNCTIONS_X1_AND_X2} displays the functions $x_{1}(\cdot; \Omega)$ and $x_{2}(\cdot; \Omega)$ for various values of $\tau$ and $\Omega$.
	\begin{figure}[h!]
		\centering
		\includegraphics[scale = 0.4]{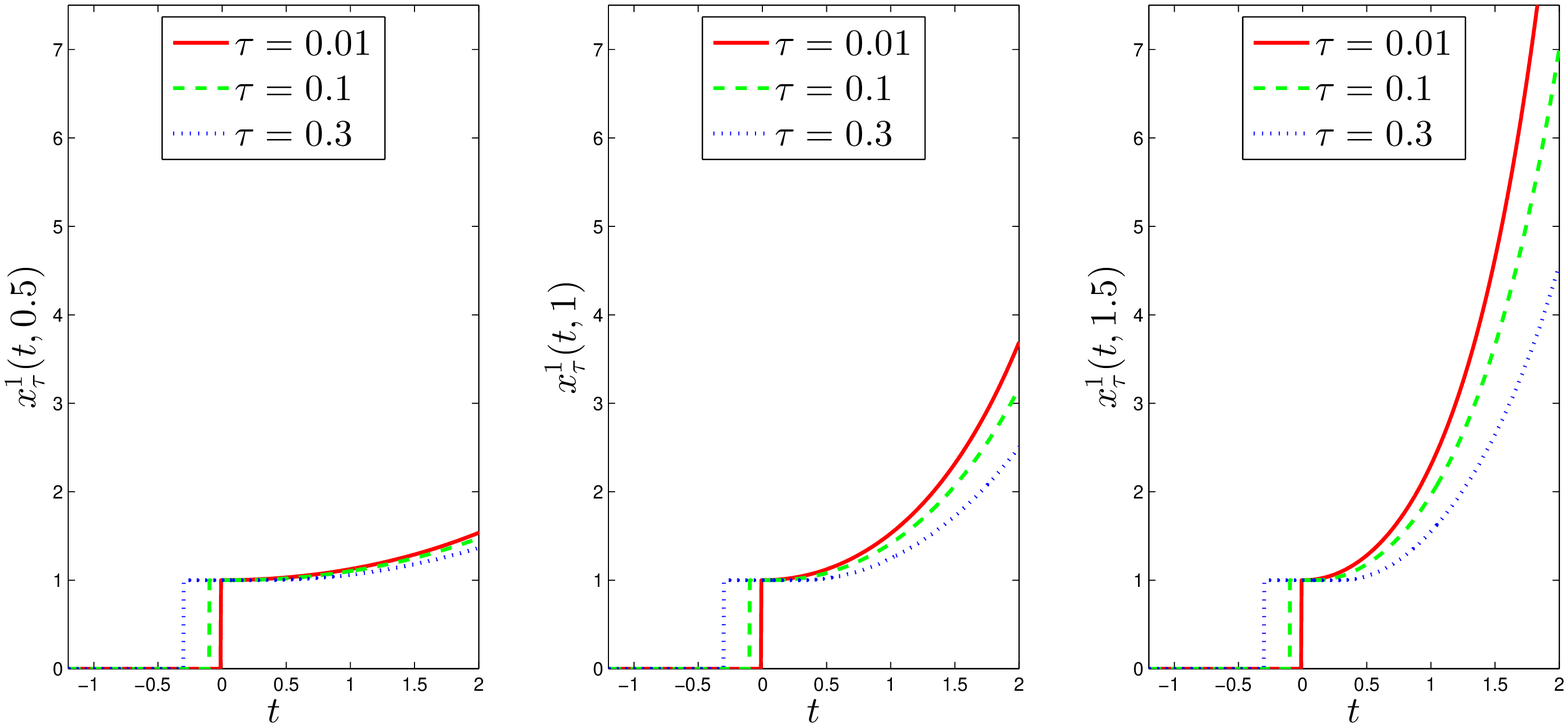}
		\includegraphics[scale = 0.4]{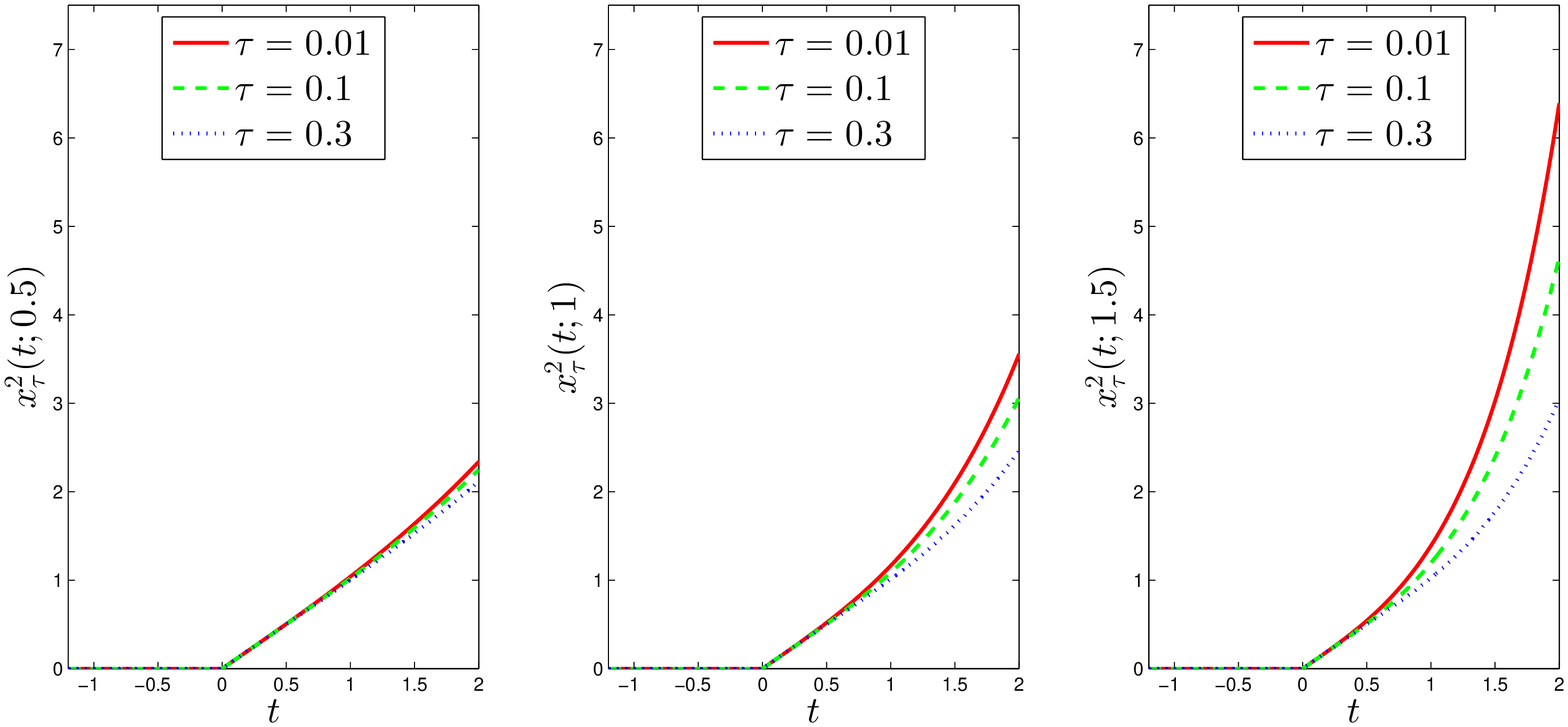}
		\caption{Functions $x_{\tau}^{1}(\cdot; \Omega)$ and $x_{\tau}^{2}(\cdot; \Omega)$. \label{FIGURE_FUNCTIONS_X1_AND_X2}}
	\end{figure}

	\begin{theorem}
		The functions $x_{1}(\cdot; \Omega), x_{2}(\cdot; \Omega)$ satisfy
		$x_{1}(\cdot; \Omega), x_{2}(\cdot; \Omega) \in C^{1}\big([-\tau, \infty), X\big) \cap C^{2}\big([-\tau, 0], X) \cap C^{2}\big([\tau, \infty), X\big)$.
		Further, $x_{1}(\cdot; \Omega)$ and $x_{2}(\cdot; \Omega)$ are solutions to the Cauchy problem (\ref{EQUATION_DDE_FOR_DEXP_1})--(\ref{EQUATION_DDE_FOR_DEXP_2})
		with the initial data
		$\varphi(t) = \mathrm{id}_{X}$, $-\tau \leq t \leq \tau$, and
		$\varphi(t) = \mathrm{id}_{X} t$, $-\tau \leq t \leq \tau$, respectively.
	\end{theorem}

	First, assuming $f \equiv 0_{X}$, Equations (\ref{EQUATION_LINEAR_OSCILLATOR_WITH_DELAY_GENERAL})--(\ref{EQUATION_LINEAR_OSCILLATOR_WITH_DELAY_GENERAL_IC})
	reduce to
	\begin{align}
		\ddot{x}(t) - \Omega^{2} x(t - 2\tau) &= 0 \text{ for } t \geq 0 \label{EQUATION_LINEAR_OSCILLATOR_WITH_DELAY_RIGHT_HAND_SIDE_ZERO}, \\
		x(t) &= \varphi(t) \text{ for } t \in [-2\tau, 0], \label{EQUATION_LINEAR_OSCILLATOR_WITH_DELAY_RIGHT_HAND_SIDE_ZERO_IC}
	\end{align}

	\begin{theorem}
		\label{THEOREM_SOLUTION_RIGHT_HAND_SIDE_ZERO}
		Let $\varphi \in C^{2}\big([-2\tau, 0], X\big)$.
		Then the unique classical solution $x$ to Cauchy problem (\ref{EQUATION_LINEAR_OSCILLATOR_WITH_DELAY_RIGHT_HAND_SIDE_ZERO})--(\ref{EQUATION_LINEAR_OSCILLATOR_WITH_DELAY_RIGHT_HAND_SIDE_ZERO_IC}) is given by
		\begin{equation}
			x(t) = x^{1}_{\tau}(t + \tau; \Omega) \varphi(-2\tau) + x^{2}_{\tau}(t + 2\tau; \Omega) \dot{\varphi}(-2\tau) +
			\int_{-2\tau}^{0} x^{2}_{\tau}(t - s; \Omega) \ddot{\varphi}(s) \mathrm{d}s. \notag
		\end{equation}
	\end{theorem}

	\begin{proof}
		To solve Equations (\ref{EQUATION_LINEAR_OSCILLATOR_WITH_DELAY_GENERAL})--(\ref{EQUATION_LINEAR_OSCILLATOR_WITH_DELAY_GENERAL_IC}), we use the ansatz
		\begin{equation}
			x(t) = x^{1}_{\tau}(t + \tau; \Omega) c_{1} + x^{2}_{\tau}(t + 2\tau; \Omega) c_{2} + \int_{-2\tau}^{0} x^{2}_{\tau}(t - s; \Omega) \ddot{c}(s) \mathrm{d}s
			\label{EQUATION_ANSATZ_X_RIGHT_HAND_SIDE_ZERO}
		\end{equation}
		for some $c_{1}, c_{2} \in X$ and $c \in C^{2}\big([-2\tau, 0], X\big)$.

		Plugging the ansatz from Equation (\ref{EQUATION_ANSATZ_X_RIGHT_HAND_SIDE_ZERO}) into Equation (\ref{EQUATION_LINEAR_OSCILLATOR_WITH_DELAY_RIGHT_HAND_SIDE_ZERO}), we obtain for $t \geq 0$
		\begin{align*}
			\frac{\mathrm{d}^{2}}{\mathrm{d}t^{2}} 
			&\Big(x^{1}_{\tau}(t + \tau; \Omega) c_{1} + x^{2}_{\tau}(t + 2\tau; \Omega) c_{2} + \int_{-2\tau}^{0} x^{2}_{\tau}(t - s; \Omega) \ddot{c}(s) \mathrm{d}s\Big) \\
			-&\Omega^{2} \Big(x^{1}_{\tau}((t + \tau) - 2\tau; \Omega) c_{1} + x^{2}_{\tau}((t + 2\tau) - 2\tau; \Omega) c_{2} \\
			&+ \int_{-2\tau}^{0} x^{2}_{\tau}((t - 2\tau) - s; \Omega) \ddot{c}(s) \mathrm{d}s = 0
		\end{align*}
		or, equivalently,
		\begin{align*}
			\Big(\frac{\mathrm{d}^{2}}{\mathrm{d}t^{2}} &x^{1}_{\tau}(t + \tau; \Omega) - \Omega^{2} x^{1}_{\tau}((t + \tau) - 2\tau; \Omega)\Big) c_{1} \\
			&+ \Big(\frac{\mathrm{d}^{2}}{\mathrm{d}t^{2}} x^{2}_{\tau}(t + 2\tau; \Omega) - \Omega^{2} x^{2}_{\tau}((t + 2\tau) - 2\tau; \Omega)\Big) c_{2} \\
			&+ \int_{0}^{\tau} \Big(\frac{\mathrm{d}^{2}}{\mathrm{d}t^{2}} x^{2}_{\tau}(t - s; \Omega) - \Omega^{2} x^{2}_{\tau}((t - 2\tau) - s; \Omega)\Big) \ddot{c}(s) \mathrm{d}s \equiv 0_{X}. \notag
		\end{align*}
		Since $x^{1}_{\tau}(\cdot; \Omega)$ and $x^{2}_{\tau}(\cdot; \Omega)$ solve the homogeneous equation,
		all three coefficients at $c_{1}$, $c_{2}$ and $\ddot{c}$ vanish
		implying that the function $x$ in Equation (\ref{EQUATION_ANSATZ_X_RIGHT_HAND_SIDE_ZERO})
		is a solution of Equation (\ref{EQUATION_LINEAR_OSCILLATOR_WITH_DELAY_RIGHT_HAND_SIDE_ZERO}).

		Now, we show that selecting $c_{1} := \varphi(-2\tau)$, $c_{2} := \dot{\varphi}(-2\tau)$ and $c := \varphi$,
		the function $x$ in Equation (\ref{EQUATION_ANSATZ_X_RIGHT_HAND_SIDE_ZERO}) satisfies the initial condition (\ref{EQUATION_LINEAR_OSCILLATOR_WITH_DELAY_RIGHT_HAND_SIDE_ZERO_IC}).
		Letting for $t \in [-2\tau, 0]$
		\begin{equation}
			\big[I \varphi\big](t) := \int_{-2\tau}^{0} x^{2}_{\tau}(t - s; \Omega) \ddot{\varphi}(s) \mathrm{d}s \notag
		\end{equation}
		and performing a change of variables $\sigma := t - s$, we find
		\begin{equation}
			\big[I \varphi\big](t) =
			\int_{t + 2\tau}^{t} x^{2}_{\tau}(\sigma; \Omega) \ddot{\varphi}(t - \sigma) \mathrm{d}\sigma
			= -\int_{t}^{t + 2\tau} x^{2}_{\tau}(\sigma; \Omega) \ddot{\varphi}(t - \sigma) \mathrm{d}\sigma. \notag
		\end{equation}
		Since $x_{2}$ can continuously be extended by $0_{L(X)}$ onto $(-\infty, -\tau]$, we get
		\begin{equation}
			\big[I \varphi\big](t) = -\int_{0}^{t + 2\tau} x_{2}(\sigma; \Omega) \ddot{\varphi}(t - \sigma) \mathrm{d}\sigma. \notag
		\end{equation}
		Integrating by parts, we further get
		\begin{equation}
			\begin{split}
				\big[I \varphi\big](t) &=
				-\int_{0}^{t + 2\tau} x^{2}_{\tau}(\sigma; \Omega) \ddot{\varphi}(t - \sigma) \mathrm{d}\sigma \\
				&= -x^{2}_{\tau}(\sigma; \Omega) \dot{\varphi}(t - \sigma)\big|_{\sigma = 0}^{\sigma = t + 2\tau} +
				\int_{0}^{t + 2\tau} \dot{x}^{2}_{\tau}(\sigma; \Omega) \dot{\varphi}(t - \sigma) \mathrm{d}\sigma.
			\end{split}
			\notag
		\end{equation}
		Now, taking into account
		\begin{equation}
			x^{2}_{\tau}(t; \Omega) = t \, \mathrm{id}_{X}, 0 \leq t \leq 2\tau, \label{EQUATION_DEFINITION_OF_X_2_BETWEEN_0_AND_2TAU}
		\end{equation}
		we obtain
		\begin{equation}
			\big[I \varphi\big](t) = -x^{2}_{\tau}(t + 2\tau; \Omega) \dot{\varphi}(-2\tau) + \int_{t}^{t + 2\tau} \dot{x}^{2}_{\tau}(\sigma; \Omega) \dot{\varphi}(t - \sigma) \mathrm{d}\sigma. \notag
		\end{equation}
		Again, using Equation (\ref{EQUATION_DEFINITION_OF_X_2_BETWEEN_0_AND_2TAU}), we compute
		\begin{equation}
			\big[I \varphi](t) = -t \dot{\varphi}(-2\tau) - \varphi(t - \sigma)\big|_{\sigma = t}^{\sigma = t + 2\tau}
			= -x^{2}_{\tau}(t; \Omega) \dot{\varphi}(-2\tau) - \varphi(-2\tau) + \varphi(t). \notag
		\end{equation}
		Hence, for $t \in [-2\tau, 0]$, we have
		\begin{equation}
			x(t) = x^{1}_{\tau}(t + \tau; \Omega) \varphi(-2\tau) + x^{2}_{\tau}(t + 2\tau; \Omega) \dot{\varphi}(-2\tau) + \int_{-2\tau}^{0} x^{2}_{\tau}(t - s; \Omega) \ddot{\varphi}(s) \mathrm{d}s =
			\varphi(t) \notag
		\end{equation}
		as claimed.
	\end{proof}

	Next, we consider Equations (\ref{EQUATION_LINEAR_OSCILLATOR_WITH_DELAY_GENERAL})--(\ref{EQUATION_LINEAR_OSCILLATOR_WITH_DELAY_GENERAL_IC}) for the trivial initial data, i.e.,
	\begin{align}
		\ddot{x}(t) - \Omega^{2} x(t - 2\tau) &= f(t) \text{ for } t \geq 0 \label{EQUATION_LINEAR_OSCILLATOR_WITH_DELAY_INITIAL_DATA_ZERO}, \\
		x(t) &= 0 \text{ for } t \in [-2\tau, 0], \label{EQUATION_LINEAR_OSCILLATOR_WITH_DELAY_INITIAL_DATA_ZERO_IC}
	\end{align}

	\begin{theorem}
		\label{THEOREM_SOLUTION_INITIAL_DATA_ZERO}
		Let $f \in C^{0}\big([0, \infty), X\big)$.
		The unique classical solution $x$ to Cauchy problem (\ref{EQUATION_LINEAR_OSCILLATOR_WITH_DELAY_INITIAL_DATA_ZERO})--(\ref{EQUATION_LINEAR_OSCILLATOR_WITH_DELAY_INITIAL_DATA_ZERO_IC}) is given by
		\begin{equation}
			x(t) = \int_{0}^{t} x^{2}_{\tau}(t - s; \Omega) f(s) \mathrm{d}s. \notag
		\end{equation}
	\end{theorem}

	\begin{proof}
		To find an explicit solution representation, we use the ansatz
		\begin{equation}
			x(t) = \int_{0}^{t} x_{2}(t - s; \Omega) c(s) \mathrm{d}s \text{ for } t \geq \tau \notag
		\end{equation}
		for some function $c \in C^{0}\big([0, \infty), X\big)$.
		Differentiating this expression with respect to $t$ and exploiting the initial conditions for $x^{2}_{\tau}(\cdot; \Omega)$, we get
		\begin{align*}
			\dot{x}(t) &= \int_{0}^{t} \dot{x}^{2}_{\tau}(t - s; \Omega) c(s) \mathrm{d}s + x^{2}_{\tau}(t - s; \Omega) c(s)\big|_{s = t}
			= \int_{0}^{t} \dot{x}^{2}_{\tau}(t - s; \Omega) c(s) \mathrm{d}s + x_{2}(0) c(t) \\
			&= \int_{0}^{t} \dot{x}^{2}_{\tau}(t - s; \Omega) c(s) \mathrm{d}s. \notag
		\end{align*}
		Differentiating again, we find
		\begin{align*}
			\ddot{x}(t) &= \int_{0}^{t} \ddot{x}^{2}_{\tau}(t - s; \Omega) c(s) \mathrm{d}s + \dot{x}^{2}_{\tau}(t - s; \Omega) c(s)\big|_{s = t} \\
			&= \int_{0}^{t} \ddot{x}^{2}_{\tau}(t - \tau - s; \Omega) c(s) \mathrm{d}s + \dot{x}^{2}_{\tau}(0+; \Omega) c(t) \\
			&= \int_{0}^{t} \ddot{x}^{2}_{\tau}(t - s; \Omega) c(s) \mathrm{d}s + c(t).
		\end{align*}
		Plugging this into Equation (\ref{EQUATION_LINEAR_OSCILLATOR_WITH_DELAY_INITIAL_DATA_ZERO})
		and recalling that $x^{2}_{\tau}(\Omega; \Omega)$ is a solution of the homogeneous equation, we get
		\begin{equation}
			c(t) \int_{0}^{t} \big(\ddot{x}^{2}_{\tau}(t - s; \Omega) - \Omega^{2} x^{2}_{\tau}(t - 2\tau - s; \Omega)\big) c(s) \mathrm{d}s = f(t) \notag
		\end{equation}
		and therefore $c \equiv f$.
	\end{proof}

	As a consequence from Theorems \ref{THEOREM_SOLUTION_RIGHT_HAND_SIDE_ZERO} and \ref{THEOREM_SOLUTION_INITIAL_DATA_ZERO}, we obtain
	using the linearity property of Equations (\ref{EQUATION_LINEAR_OSCILLATOR_WITH_DELAY_GENERAL})--(\ref{EQUATION_LINEAR_OSCILLATOR_WITH_DELAY_GENERAL_IC}):
	\begin{theorem}
		\label{THEOREM_REPRESENTATION_OF_CLASSICAL_SOLUTIONS}
		Let $\varphi \in C^{2}\big([-2\tau, 0], X\big)$ and $f \in C^{0}\big([0, \infty), X\big)$.
		The unique classical solution to Equations (\ref{EQUATION_LINEAR_OSCILLATOR_WITH_DELAY_GENERAL})--(\ref{EQUATION_LINEAR_OSCILLATOR_WITH_DELAY_GENERAL_IC}) is given by
		\begin{equation}
			\begin{split}
				x(t) &= x^{1}_{\tau}(t + \tau; \Omega) \varphi(-2\tau) + x^{2}_{\tau}(t + 2\tau; \Omega) \dot{\varphi}(-2\tau) + \int_{-2\tau}^{0} x^{2}_{\tau}(t - s; \Omega) \ddot{\varphi}(s) \mathrm{d}s \\
				&+ \left\{
				\begin{array}{cl}
					0, & t \in [-2\tau, 0), \\
					\int_{0}^{t} x^{2}_{\tau}(t - s; \Omega) f(s) \mathrm{d}s, & t \geq 0
				\end{array}\right.
			\end{split}
			\notag
		\end{equation}
		for $t \in [-2\tau, \infty)$.
	\end{theorem}

	Finally, we get:
	\begin{theorem}
		Let $\varphi \in C^{1}\big([-2\tau, 0], X\big)$ and $f \in L^{1}_{\mathrm{loc}}(0, \infty; X)$.
		The unique mild solution to Equations (\ref{EQUATION_LINEAR_OSCILLATOR_WITH_DELAY_GENERAL})--(\ref{EQUATION_LINEAR_OSCILLATOR_WITH_DELAY_GENERAL_IC}) is given by
		\begin{equation}
			\begin{split}
				x(t) &= x^{1}_{\tau}(t + \tau; \Omega) \varphi(-2\tau) + x^{2}_{\tau}(t + 2\tau; \Omega) \dot{\varphi}(0) - \int_{-2\tau}^{0} \dot{x}^{2}_{\tau}(t - s; \Omega) \dot{\varphi}(s) \mathrm{d}s \\
				&+ \left\{
				\begin{array}{cl}
					0, & t \in [-2\tau, 0), \\
					\int_{0}^{t} x^{2}_{\tau}(t - s; \Omega) f(s) \mathrm{d}s, & t \geq 0
				\end{array}\right.
			\end{split}
			\notag
		\end{equation}
		for $t \in [-2\tau, \infty)$.
	\end{theorem}

	\begin{proof}
		Approximating $\varphi$ in $C^{1}\big([-2\tau, 0], X\big)$ with $(\varphi_{n})_{n \in \mathbb{N}} \subset C^{2}\big([-2\tau, 0], X\big)$ and
		$f$ in $L^{1}_{\mathrm{loc}}(0, \infty; X)$ with $(f_{n})_{n \in \mathbb{N}} \subset C^{0}\big([0, \infty), X\big)$,
		applying Theorem \ref{THEOREM_REPRESENTATION_OF_CLASSICAL_SOLUTIONS} to solve the Cauchy problem (\ref{EQUATION_LINEAR_OSCILLATOR_WITH_DELAY_GENERAL})--(\ref{EQUATION_LINEAR_OSCILLATOR_WITH_DELAY_GENERAL_IC})
		for the right-hand side $f$ and the initial data $\varphi_{n}$, performing a partial integration for the integral involving $\ddot{\varphi}_{n}$
		and passing to the limit as $n \to \infty$, the claim follows.
	\end{proof}

	\subsection{Asymptotic behavior as $\tau \to 0$}
	Again, we assume $X$ to be a Banach space and prove the following generalization of \cite[Lemma 4]{KhuPo2014}.
	\begin{lemma}
		\label{LEMMA_DELAYED_EXPONENTIAL_ASYMPTOTICS}
		Let $\Omega \in L(X)$, $T > 0$, $\tau_{0} > 0$ and let
		\begin{equation}
			\alpha := 1 + \|\Omega\|_{L(X)} \exp\big(\tau_{0} \|\Omega\|_{L(X)}\big). \notag
		\end{equation}
		Then for any $\tau \in (0, \tau_{0}]$,
		\begin{equation}
			\|\exp_{\tau}(t - \tau; \Omega) - \exp(\Omega t)\|_{L(X)} \leq \tau \exp(\alpha T \|\Omega\|_{L(X)}) \text{ for } t \in [0, T]. \notag
		\end{equation}
	\end{lemma}

	\begin{proof}
		Let $\tau \in (0, \tau_{0}]$. For $t \in [0, \tau]$, the claim easily follows from the mean value theorem for Bochner integration.
		Next, we want to exploit the mathematical induction to show for any $k \in \mathbb{N}$.
		\begin{equation}
			\|\exp_{\tau}(t - \tau; \Omega) - \exp(t \Omega)\|_{L(X)} \leq \tau \exp\big(\alpha k \tau \|\Omega\|_{L(X)}\big) \text{ for } t \in ((k - 1) \tau, k \tau].
			\notag
		\end{equation}
		Indeed, assuming that the claim is true for some $k \in \mathbb{N}$,
		we use the fundamental theorem of calculus and find for $t \in (k \tau, (k + 1) \tau]$
		\begin{align*}
			\|&\exp_{\tau}(t - \tau; \Omega) - \exp(t \Omega)\|_{L(X)} \\
			&\leq \tau \exp\big(\alpha k \tau \|\Omega\|_{L(X)}\big) +
			\int_{k\tau}^{(k + 1) \tau} \Big\|\frac{\mathrm{d}}{\mathrm{d}s} \exp_{\tau}(s - \tau; \Omega) - \frac{\mathrm{d}}{\mathrm{d}s} \exp(s \Omega)\Big\|_{L(X)} \mathrm{d}s \\
			&\leq \tau \exp\big(\alpha k \tau \|\Omega\|_{L(X)}\big) + \|\Omega\|_{L(X)} \int_{k\tau}^{(k + 1) \tau} \|\exp_{\tau}(s - 2 \tau; \Omega) - \exp(s \Omega)\|_{L(X)} \mathrm{d}s \\
			&\leq \tau \exp\big(\alpha k \tau \|\Omega\|_{L(X)}\big) + \|\Omega\|_{L(X)} \int_{k\tau}^{(k + 1) \tau} \big\|\exp_{\tau}(s - 2 \tau; \Omega) - \exp\big((s - \tau) \Omega\big)\big\|_{L(X)} \mathrm{d}s \\
			&+ \|\Omega\|_{L(X)} \int_{k\tau}^{(k + 1) \tau} \big\|\exp(s\Omega) - \exp\big((s - \tau) \Omega\big)\big\|_{L(X)} \mathrm{d}s \displaybreak \\
			&\leq \tau \exp\big(\alpha k \tau \|\Omega\|_{L(X)}\big) + \|\Omega\|_{L(X)} \int_{(k - 1) \tau}^{k \tau} \big\|\exp_{\tau}(s - \tau; \Omega) - \exp(s \Omega)\big\|_{L(X)} \mathrm{d}s \\
			&+ \|\Omega\|_{L(X)} \int_{k\tau}^{(k + 1) \tau} \int_{s - \tau}^{s} \Big\|\frac{\mathrm{d}}{\mathrm{d}\sigma} \exp(\sigma \Omega)\Big\|_{L(X)} \mathrm{d}\sigma \mathrm{d}s \\
			&\leq \tau \exp\big(\alpha k \tau \|\Omega\|_{L(X)}\big) + \tau^{2} \|\Omega\|_{L(X)} \exp\big(\alpha k \tau \|\Omega\|_{L(X)}\big) \\
			&+ \tau^{2} \|\Omega\|_{L(X)}^{2} \exp\big((k + 1) \tau \|\Omega\|_{L(X)}\big) \\
			&\leq \tau \exp\big(\alpha k \tau \|\Omega\|_{L(X)}\big) \Big(1 + \tau \|\Omega\|_{L(X)} + \tau \|\Omega\|_{L(X)}^{2} \exp\big(\tau \|\Omega\|_{L(X)}\big)\Big)  \\
			&\leq \tau \exp\big(\alpha k \tau \|\Omega\|_{L(X)}\big) \bigg(1 + \tau \|\Omega\|_{L(X)} \Big(1 + \tau \|\Omega\|_{L(X)} \exp\big(\tau \|\Omega\|_{L(X)}\big)\Big)\bigg) \\
			&\leq \tau \exp\big(\alpha k \tau \|\Omega\|_{L(X)}\big) \exp(\alpha \tau \|\Omega\|_{L(X)}\big) \leq \exp\big(\alpha (k + 1) \tau \|\Omega\|_{L(X)}\big)
		\end{align*}
		since $\alpha \geq 1$.
		The claim follows by induction.
	\end{proof}

	\begin{corollary}
		\label{COROLLARY_DELAYED_EXPONENTIAL_ASYMPTOTICS}
		Let the assumptions of Lemma \ref{LEMMA_DELAYED_EXPONENTIAL_ASYMPTOTICS} be satisfied and let $\gamma \geq 0$. Then
		\begin{equation}
			\big\|\exp_{\tau}(t + \gamma; \Omega) - e^{\Omega t}\big\|_{L(X)}
			\leq (\gamma + \tau) \big(1 + \|\Omega\|_{L(X)}\big) \exp\big(\alpha (T + \gamma + \tau) \|\Omega\|_{L(X)}\big). \notag
		\end{equation}
	\end{corollary}

	\begin{proof}
		Lemma \ref{LEMMA_DELAYED_EXPONENTIAL_ASYMPTOTICS} and the mean value theorem for Bochner integration yield
		\begin{align*}
			\big\|\exp_{\tau}(t + &\gamma; \Omega) - e^{\Omega t}\big\|_{L(X)} \\
			&\leq \big\|\exp_{\tau}(t + \gamma; \Omega) - e^{\Omega (t + \gamma + \tau)}\big\|_{L(X)} +
			\big\|e^{\Omega (t + \gamma + \tau)} - e^{\Omega t}\big\|_{L(X)} \\
			&\leq \tau \exp\big(\alpha (T + \gamma + \tau) \|\Omega\|_{L(X)}\big) + 
			(\gamma + \tau) \|\Omega\|_{L(X)} \exp\big((T + \gamma + \tau) \|\Omega\|_{L(X)}\big) \\
			&\leq (\gamma + \tau) \big(1 + \|\Omega\|_{L(X)}\big) \exp\big(\alpha (T + \gamma + \tau) \|\Omega\|_{L(X)}\big)
		\end{align*}
		as we claimed.
	\end{proof}

	Let $T > 0$, $\tau_{0} > 0$, $x_{0}, x_{1} \in X$ and $f \in L^{1}_{\mathrm{loc}}(0, \infty; X)$ be fixed and
	let $\bar{x} \in C^{1}\big([0, \infty), X\big)$ denote the unique mild solution to the Cauchy problem 
	(\ref{EQUATION_HARMONIC_OSCILLATOR_GENERAL})--(\ref{EQUATION_HARMONIC_OSCILLATOR_GENERAL_IC})
	from Section \ref{SECTION_CLASSICAL_HARMONIC_OSCILLATOR}.
	\begin{theorem}
		\label{THEOREM_DELAYED_EQUATION_ASYMPTOTICS}
		Let $\tau > 0$.	For any $\tau \in (0, \tau_{0})$, let $x(\cdot; \tau)$ denote the unique mild solution of
		(\ref{EQUATION_LINEAR_OSCILLATOR_WITH_DELAY_GENERAL})--(\ref{EQUATION_LINEAR_OSCILLATOR_WITH_DELAY_GENERAL_IC})
		for the initial data $\varphi(\cdot; \tau) \in C^{1}\big([-2\tau, 0], X\big)$.
		Then we have
		\begin{equation}
			\begin{split}
				\|x(\cdot; \tau) - \bar{x}\|_{C^{0}([0, T], X)} &\leq 3 \beta \Big(
				\|\varphi(-2\tau; \tau) - x_{0}\|_{X} + \|\dot{\varphi}(0; \tau) - x_{1}\|_{X}\Big) \\
				&+ 3 \beta \tau \Big(\|\varphi(\cdot; \tau)\|_{C^{1}([-2\tau, 0], X)} + \|f\|_{L^{1}(0, T; X)}\Big)
			\end{split}
			\notag
		\end{equation}
		with $\beta(T) := 2 \big(1 + \|\Omega\|_{L(X)}\big) \big(1 + \|\Omega^{-1}\|_{L(X)}\big) \exp\big(\alpha (T + 2\tau) \|\Omega\|_{L(X)}\big)$.
	\end{theorem}

	\begin{proof}
		Using the explicit representation of $\bar{x}$ and $x(\cdot; \tau)$ and $x$ from Sections \ref{SECTION_CLASSICAL_HARMONIC_OSCILLATOR} and \ref{SECTION_HARMONIC_OSCILLATOR_WITH_DELAY_SOLUTION_REPRESENTATION},
		respectively, we can estimate
		\begin{equation}
			\|x(t; \tau) - \bar{x}(t)\|_{X} \leq
			I_{0, 1}(t) + I_{0, 2}(t) + I_{0, 2}(t) \text{ for } t \in [0, T] \notag
		\end{equation}
		with
		\begin{align*}
			I_{0, 1}(t) &:= \big\|x_{\tau}^{1}(t + \tau; \Omega) \varphi(-2\tau; \tau) - \tfrac{1}{2} (e^{\Omega t} + e^{-\Omega t}) x_{0}\big\|_{X} \\
			&+ \big\|x_{\tau}^{2}(t + 2\tau; \Omega) \dot{\varphi}(0; \tau) + \tfrac{1}{2} \Omega^{-1}(e^{\Omega t} - e^{-\Omega t}) x_{1}\big\|_{X}, \\
			I_{0, 2}(t) &:= \int_{0}^{t} \big\|x^{2}_{\tau}(t - s; \Omega) - \tfrac{1}{2} \Omega^{-1} (e^{\Omega(t - s)} - e^{-\Omega (t -s)})\big\|_{L(X)} \|f(s)\|_{X} \mathrm{d}s, \\
			I_{0, 3}(t) &:= \int_{-2\tau}^{0} \|x_{\tau}^{2}( t - s - \tau; \Omega)\|_{L(X)} \|\dot{\varphi}(s; \tau)\|_{X} \mathrm{d}s.
		\end{align*}

		Corollary \ref{COROLLARY_DELAYED_EXPONENTIAL_ASYMPTOTICS} yields
		\begin{align*}
			\big\|x_{\tau}^{1}(t + \tau; \Omega) - \tfrac{1}{2} (e^{\Omega t} + e^{-\Omega t})\big\|_{L(X)} &\leq \beta \tau, \\
			\big\|x_{\tau}^{2}(t + \tau; \Omega) - \tfrac{1}{2} \Omega^{-1} (e^{\Omega t} - e^{-\Omega t})\big\|_{L(X)} &\leq \beta \tau
		\end{align*}
		and, therefore,
		\begin{align*}
			I_{0, 1}(t) &\leq \beta \tau \big(\|\varphi(-2\tau; \tau)\|_{X} + \|\dot{\varphi}(0; \tau)\|_{X}\big) + \beta\big(\|\varphi(-2\tau; \tau) - x_{0}\|_{X} + \|\dot{\varphi}(0; \tau) - x_{1}\|_{X}\big) \\
			&\leq \beta \tau \|\varphi\|_{C^{1}([-2\tau, 0], X)} + \beta \big(\|\varphi(-2\tau; \tau) - x_{0}\|_{X} + \|\dot{\varphi}(0; \tau) - x_{1}\|_{X}\big).
		\end{align*}
		Similarly,
		\begin{equation}
			I_{0, 2}(t) \leq 2 \beta \tau \|f\|_{L^{1}(0, T; X)} \text{ and }
			I_{0, 3}(t) \leq 2 \beta \tau \|\varphi\|_{C^{1}([0, T], X)}. \notag
		\end{equation}
		Hence, the claim follows.
	\end{proof}

	\begin{corollary}
		Under conditions of Theorem \ref{THEOREM_DELAYED_EQUATION_ASYMPTOTICS}, we additionally have
		\begin{equation}
			\begin{split}
				\|x(\cdot; \tau) - &\bar{x}\|_{C^{1}([0, T], X)} \leq 3 (1 + \beta(T))(1 + \delta(T)) (1 + T) \Big(
				\|\varphi(-2\tau; \tau) - x_{0}\|_{X} \\
				+ &\|\dot{\varphi}(0; \tau) - x_{1}\|_{X} + \tau\big(\|\varphi(\cdot; \tau)\|_{C^{1}([-2\tau, 0], X)} + \|f\|_{L^{1}(0, T; X)} + \|x_{0}\|_{X} + \|x_{1}\|_{X}\big)\Big)
			\end{split}
			\notag
		\end{equation}
		with $\delta(T) := \|\Omega\|_{L(X)}^{2} \big(2 + \|\Omega^{-1}\|_{L(X)} + \|\Omega^{-1}\|_{L(X)} T\big) e^{\|\Omega\|_{L(X)} T}$.
	\end{corollary}

	\begin{proof}
		Integrating Equation (\ref{EQUATION_HARMONIC_OSCILLATOR_GENERAL}) and using Equation (\ref{EQUATION_HARMONIC_OSCILLATOR_GENERAL_IC})
		as well as exploiting Equations (\ref{EQUATION_LINEAR_OSCILLATOR_WITH_DELAY_MILD_FORMULATION})--(\ref{EQUATION_LINEAR_OSCILLATOR_WITH_DELAY_MILD_FORMULATION_IC}) yields
		\begin{align*}
			\|\dot{x}(t; \tau) - \dot{\bar{x}}(t)\| &\leq 
			\|\dot{\varphi}(0; \tau) - x_{1}\|_{X} +
			\int_{0}^{t} \|\Omega^{2} x(s - 2\tau; \tau) - \Omega^{2} \bar{x}(s)\|_{X} \mathrm{d}s \\
			&\leq I_{1, 1}(t) + I_{1, 2}(t) + I_{1, 3}(t) \text{ for } t \in [0, T] \notag
		\end{align*}
		with
		\begin{align*}
			I_{1, 1}(t) &:= \|\dot{\varphi}(0; \tau) - x_{1}\|_{X}, \quad I_{1, 2} := \|\Omega\|_{L(X)}^{2} \int_{-2\tau}^{0} \|\varphi(s) - \bar{x}(s + 2\tau)\|_{X} \mathrm{d}s, \\
			I_{1, 3}(t) &:= \|\Omega\|_{L(X)}^{2} \int_{2\tau}^{t} \|x(s - 2\tau; \tau) - \bar{x}(s)\|_{X} \mathrm{d}s
		\end{align*}
		Taking into account Equation (\ref{EQUATION_HARMONIC_OSCILLATOR_GENERAL_EXPLICIT_SOLUTION}), we can estimate
		\begin{equation}
			\|\bar{x}\|_{C^{0}([0, 2\tau], X)} \leq \big(\|x_{0}\| + \|\Omega^{-1}\|_{L(X)} \|x_{1}\|\big) e^{\|\Omega\|_{L(X)} T} +
			\|\Omega^{-1}\|_{L(X)} T e^{\|\Omega\|_{L(X)} T} \|f\|_{L^{1}(0, T; X)}. \notag
		\end{equation}
		Hence,
		\begin{align*}
			I_{1, 2}(t) \leq \delta \tau \Big(\|\varphi\|_{C^{0}([0, T], X)} + \|x_{0}\|_{X} + \|x_{1}\|_{X}\Big). \notag
		\end{align*}
		Applying Theorem \ref{THEOREM_DELAYED_EQUATION_ASYMPTOTICS}, we further get
		\begin{align*}
			I_{1, 3}(t) \leq 3 \|\Omega\|_{L(X)}^{2} T \beta \Big(&\|\varphi(-2\tau; \tau) - x_{0}\|_{X} + \|\dot{\varphi}(0; \tau) - x_{1}\|_{X} \\
			+ &\tau \big(\|\varphi(\cdot; \tau)\|_{C^{1}([-2\tau, 0], X)} + \|f\|_{L^{1}(0, T; X)}\big)\Big).
		\end{align*}
		Combining these inequalities and using again Theorem Theorem \ref{THEOREM_DELAYED_EQUATION_ASYMPTOTICS}, we deduce the estimate asserted.
	\end{proof}

\section*{Acknowledgment}
This work has been funded by a research grant from the Young Scholar Fund
supported by the Deutsche Forschungsgemeinschaft (ZUK 52/2) at the University of Konstanz, Konstanz, Germany.

\end{document}